\documentclass{arxiv}

\usepackage{caption, amssymb,amsfonts,amsmath, graphicx, geometry}
\usepackage[nice]{nicefrac}
\usepackage{mathrsfs}
\usepackage{enumerate}
\usepackage{bbm}

\newcommand{\R}{\mathbb R}
\newcommand{\1}{{\mathbbm{1}}}
\newcommand{\E}{\mathbb E}
\renewcommand{\P}{\mathbb P}

\newcommand{\G}{\mathcal G}
\newcommand{\X}{\mathcal X}
   
\newcommand{\x}{\mathbf{x}}
\newcommand{\y}{\mathbf{y}}

\newcommand{\p}{\mathbf{p}}

\newcommand{\eff}{\rm{eff}}
\newcommand{\mix}{\rm{mix}}
\newcommand{\PPP}{\rm{PPP}}
\newcommand{\de}{\mathrm{d}}
\DeclareMathOperator{\Cov}{Cov}

\newtheorem{theorem}{Theorem}

\newtheorem{lemma}{Lemma}

\title [Subcriticality in weight-dependent random connection models]
{Only long edges can erase the subcritical annulus-crossing phase of weight-dependent
	random connection models}

\author[Emmanuel Jacob]{Emmanuel Jacob}

\begin{document}
\maketitle

\begin{quote}
{\small {\bf Abstract:} This short note aims at complementing the results of the recent work \cite{JL2023existence}, where Jahnel and L\"uchtrath investigate the question of existence of a subcritical percolation phase for the annulus-crossing probabilities in a large class of continuum percolation models, namely the classical or generalized weight-dependent random connection models. Their work relates the absence of a subcritical phase to the occurrence of long edges in the graph, through a somewhat indirect criterium, that stays inconclusive in some models of interest. We provide a more direct and arguably simpler criterium, that is always conclusive when considering classical weight-dependent random connection models.
}
\end{quote}

\vspace{0.5cm}

\section{Introduction}

This note focuses on the existence of a subcritical annulus-crossing phase in continuum percolation theory. We will consider classical or generalised weight-dependent random connection models, but we first discuss the case of the Poisson-boolean model (see~\cite{meester_roy_1996,Gouere08}) to motivate our work.

\medskip
In the Poisson-boolean model, we consider a random graph $\G^\lambda$ whose vertex set is given by a Poisson point process of intensity $\lambda>0$ on $\R^d$. The vertices are given i.i.d. radii (or weights), and two vertices $x$ and $y$ with radii $R_x$ and $R_y$ are connected if $|x-y|<R_x+R_y$, namely if the balls $B(x,R_x)$ and $B(y,R_y)$ intersect. This model exhibits a phase transition at a \emph{critical percolation intensity} $\lambda_c\in [0,+\infty)$: for $\lambda>\lambda_c,$ the graph $\G^\lambda$ has a.s. a (unique) infinite component, while for $\lambda<\lambda_c$, the connected components of $\G^\lambda$ are a.s. all finite. Of course, this is a true phase transition only if there exists a subcritical phase, namely if $\lambda_c>0$. This subcritical phase actually fails to exist if $\E[R^d]=+\infty$, which is a degenerate case in the sense that the vertices almost surely all have infinite degree, whatever the value of $\lambda>0$ (see~\cite{meester_roy_1996}). However, Gou\'er\'e proved in~\cite{Gouere08} that $\lambda_c>0$ whenever $\E[R^d]<+\infty$, showing the existence of a subcritical phase transition for all nondegenerate models. 
His proof relies on a renormalization lemma for the annulus-crossing probabilities. Taking $\lambda>0$ small enough allows to initialize a simple renormalization argument that shows that these probabilities tend to 0 when the size of the annuli tend to infinity, which shows $\lambda\le \lambda_c$. His technique also leads naturally to the introduction of the \emph{critical annulus-crossing intensity} $\widehat \lambda_c$, which automatically satisfies $\widehat \lambda_c\le\lambda_c$, as well as $\widehat \lambda_c>0$ by the aforementioned renormalization lemma. Interestingly, taking any $\lambda<\widehat \lambda_c$ also provides renormalization techniques that can be used to derive many properties of this phase, see~\cite{GourereTheret19}. It is thus natural to ask whether we have $\widehat \lambda_c= \lambda_c$, namely whether the two critical intensities coincide. This difficult question has been investigated in~\cite{ATT17,dembin2022sharp,DCRT2020}, and it is proven in~\cite{dembin2022sharp} that indeed $\widehat \lambda_c=\lambda_c$ when $R$ follows a power-law, except possibly for a countable number of exceptional power-law distributions.

\medskip

 We are however interested in a larger class of models that can feature long edges and scale-free degree distribution, namely the classical weight-dependent random connection models as in~\cite{gracar2022recurrence}, or a generalized version in~\cite{JL2023existence}. We first discuss the classical settings, which already include various previously studied models like long-range percolation, scale-free percolation or geometric inhomogeneous random graphs, and are closely related to hyperbolic random geometric graph, see also~\cite{jorritsma2023clustersize} for a recent work on such models. 
The vertices still have i.i.d. weights, but the probability of finding an edge between a vertex $x$ of weight $W_x$ and a vertex $y$ of weight $W_y$ is now of the form
\[
\varphi(W_x,W_y,|x-y|),
\]
for $\varphi$ some profile function satisfying some mild assumptions ensuring the degrees are finite. Actually~\cite{gracar2022recurrence} use a more specific form
\[
\varphi(W_x,W_y,|x-y|)=\rho(g(W_x,W_y)|x-y|^d),
\]
but we will not require this form in our work.
For these models, we still can naturally define $\lambda_c$ and $\widehat \lambda_c$, that satisfy $0\le \widehat \lambda_c\le \lambda_c$. The natural questions are now whether we have equalities $\lambda_c=0$, or $\widehat \lambda_c=0$, or $\lambda_c=\widehat \lambda_c$. The question of whether $\lambda_c=0$ has been investigated in various works, see in particular~\cite{gracar2021percolation} for a fairly general result. 

In~\cite{JL2023existence}, it has been  observed the interesting phenomenon that one may very-well have $0=\widehat \lambda_c<\lambda_c$. This can be viewed as a ``weak form of degeneracy'' for these models, that implies $\widehat \lambda_c=0$ without necessarily saying anything about $\lambda_c$. More precisely, they relate the property $\widehat \lambda_c=0$ to the occurence of long edges in the graph, as measured through a coefficient $\delta_{\eff}\ge 0$. In short, they obtain that if $\delta_{\eff}>2$, then long edges are rare and $\widehat \lambda_c>0$, while if $\delta_{\eff}<2$, long edges are common and $\widehat \lambda_c=0$. Their results leave the case $\delta_{\eff}=2$ open, which however includes some models of interest and in particular the age-dependent random connection model of~\cite{GGLM19_AgeRCM}.

We elaborate on~\cite{JL2023existence} by relating more directly the property $\widehat \lambda_c=0$ to the occurence of long edges in the graph. Our criterium always concludes whether we have $\widehat \lambda_c=0$ or not, at least in the case of the classical weight-dependent random connection models. Interestingly, \cite{JL2023existence}  introduce and study a generalized version of weight-dependent random connection models, where the connection probabilty between two vertices is allowed to further depend on the surrounding vertices of the graph. In this case we obtain partial results, which may complement their work. Note however that in the full generality introduced in~\cite{JL2023existence}, it is not necessarily the case that the percolation probabilities are monotonous in $\lambda$. It is thus unclear that there is a single percolation phase transition, and it seems dubious that much can be said without introducing further assumptions.

\medskip

We conclude this introduction with the following takeaway message from our work: If it is likely to find long edges in our graphs, namely edges of length of order $r$ in a ball of radius $r$, then this gives a weak form of degeneracy that implies $\widehat \lambda_c=0$ without necessarily implying $\lambda_c=0$. This leaves open the question whether we always have $\lambda_c= \widehat \lambda_c>0$ whenever this weak form of degeneracy does not occur.

\section{Framework}

We use almost the same settings and notations as~\cite{JL2023existence} to describe the generalized weight-dependent random connection model. For $\lambda>0$, we consider $\mathcal G^\lambda$ 
a random graph with vertex set $\mathcal X$ given by a Poisson Point Process (\PPP) on $\R^d\times (0,1)$ with intensity $\lambda>0$, where a vertex $\x=(x,u_x)$ is said to be located at position $x\in \R^d$ and with mark $u_x\in (0,1)$. We then connect independently each unordered pair of vertices $\x,\y \in \mathcal X$ with probability
\[
\p(\x,\y,\mathcal X \backslash\{\x,\y\}),
\]
where $\p$ is some functional invariant under translations and rotations. It further satisfies the following integrability condition: for any $\x$ and $\y$ in $\R^d\times (0,1)$, we have
\begin{equation}
\E_\lambda[\p(\x,\y,\X)]=\varphi_\lambda(u_x,u_y,|x-y|)
\end{equation}
for some measurable function $\varphi_\lambda:(0,1)\times (0,1)\times (0,+\infty)$. Note $\p$ does not depend on $\lambda$, but $\varphi_\lambda$ can depend on $\lambda$ because of the expectation taken with respect of a $\PPP$ with intensity $\lambda$. We suppose:
\begin{enumerate}
	\item $\varphi_\lambda$ is symmetric in the first two arguments and non-increasing in the third.
	\item The integral
	\[\int_0^1 \int_0^1 \int_{\R^d} \varphi_\lambda(s,t,|z|) \de z \de s \de t\]
	is finite.
\end{enumerate}
We will sometimes introduce further hypotheses
\begin{itemize}
	\item Mixing hypothesis: We just repeat the hypothesis introduced in~\cite{JL2023existence} to 
	 specify local percolation and quantify the influence of far-apart vertices on the connection mechanism. Write 
	\begin{align}
	\nonumber G(r,x)&:=\{\text{there is a path in }\G^\lambda \text{ from } B(x,r) \text{ to } B(x,2r)^\complement\\
	\nonumber& \null \qquad  \text{ using only vertices in } B(x,3r)
	\}\\
	\nonumber G(r)&:=G(r,0).
	\end{align}
	In these events we ask to find a crossing of some (large) annulus using only vertices contained in a larger ball containing the annulus. We now say that $\G^\lambda$ is mixing with index $\zeta>0$ if there exists $C_{\mix}>0$ such that for all $\lambda>0$ and all $|x|>6r$, we have
	\begin{equation} \label{ineq:mixing}
	\left|\Cov(\1_{G(r)},\1_{G(r,x)})\right|\le C_{\mix} \lambda r^{-\zeta}.
	\end{equation}
	\item Monotonicity hypothesis: 
	We say the model satisfies the monotonicity hypothesis if $\p$ is nondecreasing in the third argument, with the understanding that adding vertices in your $\PPP$ cannot decrease the edge probabilities of the existing vertices. 	If it holds,
	then $\varphi_\lambda$ is of course nondecreasing in $\lambda$, and more generally we can couple the graphs $\G^\lambda$ and $\G^{\lambda'}$ based on the $\PPP(\lambda)$ and $\PPP(\lambda')$ with $\lambda<\lambda'$, such that $\G^\lambda$ is a subgraph of $\G^{\lambda'}$. In particular, the percolation probability becomes nondecreasing in $\lambda$ so that we can naturally introduce the \emph{classical critical percolation intensity} $\lambda_c$, as well as the \emph{critical annulus-crossing intensity} $\widehat \lambda_c$, defined by	
	\begin{align*}
	\widehat \lambda_c&:= \sup\left\{\lambda>0, \P_\lambda\left(B(0,r)\leftrightarrow B(0,2r)^\complement\right)\to 0\right\} \\
	&= \inf\left\{\lambda>0, \limsup \P_\lambda \left(B(0,r)\leftrightarrow B(0,2r)^\complement\right)>0\right\},
	\end{align*}
	where we write $A\leftrightarrow B$ for the existence in $\G^\lambda$ of a path started at a vertex located in $A$ and ending at a vertex located in $B$. Without the monotonicity assumption, we still can define $\widehat \lambda_c$, but we cannot easily guarantee for $\lambda>\widehat \lambda_c$ that the 
	probabilities $\P_\lambda \left(B(0,r)\leftrightarrow B(0,2r)^\complement\right)$ tend to 0 as $r$ tends to infinity. Note that~\cite{JL2023existence} does not state the monotonicty hypothesis, and introduce as an example a model that does not satisfy this hypothesis.
	\item Classical weight-dependent random connection model: In the classical model, $\p$ does not depend on the surrounding vertices, so $\p(\x,\y, \X)=\p(\x,\y)= \varphi(u_x,u_y,|x-y|)$, where of course $\varphi$ does not depend on $\lambda$. This is a strong assumption, that implies both the mixing and monotonicity assumptions. This assumption is also denoted as the ``$\zeta=\infty$'' case in~\cite{JL2023existence}.
\end{itemize}

This note aims at understanding in which cases we have a subcritical phase, i.e. $\widehat \lambda_c>0$.

\section{Results}

We claim that the main reason that can lead to an absence of a subcritical phase is the presence of long edges in the graph. For $r>0$ and $c>0$, we define the ``long edge event''
\begin{equation}
L(r,c):= \left\{\exists x\sim y, |x|<r, |y-x|>cr\right\},
\end{equation}
namely the event that we can find an edge of length at least $cr$ with one endvertex in the ball of radius $r$ centered at the origin. We will also consider $H(\lambda,c)$ the property
\[
\left\{\limsup_{r\to \infty}\ \P_\lambda\left(L(r,c)\right)>0.
\right\}
\]
In words, $H(\lambda,c)$ holds if it is not unlikely to find long edges -- namely edges of length comparable to the size of the ball where we search for them -- when $r$ tends to infinity. The following simple property makes this heuristic description actually quite pertinent:
\begin{lemma}
	$H(\lambda,c)$ does not depend on $c>0$.
\end{lemma}

\begin{proof}
	Clearly $H(\lambda,c)\Rightarrow H(\lambda,c')$ for all $c'<c$. Consider now $c'>c>0$. Since the model is translation-invariant, and since the number of balls of radius $r$ required to cover a ball of radius $c'r/c$ is some number $n(c'/c)$ depending only on $c'/c$ and not on $r$, we have
	\[
	n(c'/c)\ \P_\lambda\left(\exists  x\sim y, x\in B(0,r), |y-x|>c'r\right) \ge
	\P_\lambda\left(\exists  x\sim y, x\in B(0,c'r/c), |y-x|>c'r\right).
	\]
	If $H(\lambda,c)$ holds, then the limsup of the RHS is positive, thus the limsup of the LHS is also positive, and we deduce $H(\lambda,c')$ as required. 
%
\end{proof}
In the following, we will write $H(\lambda)$ for the property $H(\lambda, 1)$ (which by the lemma is actually the same as $H(\lambda, c)$ for any given value $c>0$). Under the additional assumption that we work with a classical weight-dependent random connection model, the following property makes the assumption $H(\lambda)$ even more interesting.
\begin{lemma}\label{lem:H_does_not_depend_on_lambda}
	For the classical weight-dependent random connection models, the property $H(\lambda)$ does not depend on $\lambda>0$. 
\end{lemma}

\begin{proof}
	For the classical weight-dependent random connection models and $0<\lambda<\lambda'$, we can couple $\G^\lambda$ and $\G^{\lambda'}$, so that $\G^\lambda$ is the random subgraph of $\G^{\lambda'}$ obtained by keeping each vertex of $\G^{\lambda'}$ independently with probability $\lambda/\lambda'$. Therefore
	\[
	 (\lambda/\lambda')^2\ \P_{\lambda'}(L(r,1))\le \P_\lambda(L(r,1))\le \P_{\lambda'}(L(r,1)).
	\]
	The result follows immediately.
\end{proof}

We now state our main result, which shows that the property $H(\lambda)$ is intimately linked to the nonexistence of a subcritical phase. 
\pagebreak[3]

\begin{theorem}
	\begin{enumerate}
		\item In the classical weight-dependent random connection model, the property $H(\lambda)$ does not depend on $\lambda>0$, and is equivalent to $\widehat \lambda_c=0$.
		\item In the generalized weight-dependent random connection model:
	\begin{enumerate}
	\item If $H(\lambda)$ holds for all $\lambda>0$, then $\widehat \lambda_c=0$.
	\item If $H(\lambda)$ does not hold for some $\lambda>0$, and if additionally we make the monotonicity and mixing assumptions, then $\widehat \lambda_c>0$.	
	\end{enumerate}	
	\end{enumerate}
\end{theorem}
Part (1) of the theorem directly follows from Part (2) together with Lemma~\ref{lem:H_does_not_depend_on_lambda}, so we only need to prove Part (2). 
Part (2)(a) follows from the simple observation that the event $L(r,3)$ implies the existence of an edge between a vertex in $B(0,r)$ and a vertex at distance larger than $3r$, thus necessarily located outside $B(0,2r)$. Hence $H(\lambda,3)\Rightarrow \lambda\ge \widehat \lambda_c$.

Finally, Part (2)(b) of the theorem is a consequence of the following lemma:

\begin{lemma}
	Suppose $H(\lambda,1/20)$ does not hold, and work under the monotonicity and mixing assumptions. Then $\widehat \lambda_c>0$.
\end{lemma}

\begin{proof}
 	Recall the event $G(r)$ defined in the mixing hypothesis, and introduce the following events:
	\begin{align*}
	C(r)&:=\{B(x,r) \leftrightarrow B(x,2r)^\complement\},\\
	F(r)&:=\{\exists x\sim y, x\in B(0,20r), |y-x|>r\}.
	\end{align*}
	Using the mixing property, we obtain a renormalization lemma similar\footnote{The only difference in our renormalization lemma is that we work with the events $C(r)$ rather than with $G(r)$, which avoids us to separately have to control the probability of $C(r)\backslash G(r)$.} to the one stated in~\cite{JL2023existence}: 
	\begin{align*}
	\P_\lambda(C(10 r))&\le  C \P_\lambda(G(r))^2+ \P_\lambda(F(r))+ C C_{mix} \lambda r^{-\zeta}\\
	&\le  C \P_\lambda(C(r))^2+ \P_\lambda(F(r))+ C C_{mix} \lambda r^{-\zeta},
	\end{align*}
	where $C$ is an explicit constant depending only on the dimension. We do not detail this classical argument, but just observe that for classical weight-dependent random connection models, the last term with $C_{\mix}$ disappears, as in the original renormalization lemma of~\cite{Gouere08}. 	
	In order to use such a renormalization lemma to deduce that $\P_\lambda(G(r))$ tends to 0 as $r\to +\infty$, two ingredients are necessary:
	\begin{itemize}
		\item We need the ``error terms'' $\P_\lambda(F(r))+ C C_{mix} \lambda r^{-\zeta}$ tending to 0 as $r\to \infty$. The mixing term clearly does, while we obtain $\P_\lambda(F(r)) \to 0$ by the hypothesis that $H(\lambda, 1/20)$ does not hold.
		\item We need to initialize the renormalization argument by ensuring that $\P_\lambda(C(r))$ is small on some interval $r\in [r_0,10 r_0]$. To this end, we can simply bound roughly the probability of $C(r)$ by the probability of finding any point of the $\PPP$ in the ball $B(0,r)$, and \emph{decrease} sufficiently the value of $\lambda$ to make this probability small. While decreasing the value of $\lambda$, we use the monotonicity property to ensure the error term does not increase\footnote{Note that in models not satisfying the monotonicity assumption, we could still conclude if we have a good control of the probabilities $\P_\lambda(F(r))$, which in many cases should not be too hard to obtain.}. 
	\end{itemize}
	This classical renormalization technique thus allows to conclude.
\end{proof}

\bibliographystyle{arxiv}
\bibliography{mabiblio}

\end{document}